\documentclass[11pt,reqno]{amsart}
\usepackage{amsfonts}
\usepackage{fancyhdr}
\usepackage{times}
\usepackage[T1]{fontenc}
\usepackage{mathrsfs}
\usepackage{latexsym}
\usepackage[dvips]{graphics}
\usepackage{epsfig}
\usepackage{hyperref, amsmath, amsthm, amsfonts, amscd, flafter,epsf}
\usepackage{amsmath,amsfonts,amsthm,amssymb,amscd}

\usepackage{longtable}
\usepackage{tabu}
\usepackage{array}
\usepackage{multirow}

\input amssym.def
\input amssym.tex
\usepackage{color}
\usepackage[parfill]{parskip}
\usepackage{mathrsfs} 
\usepackage{geometry}                
\usepackage{epstopdf}

\usepackage{verbatim}

\newtheorem{thm}{Theorem}[section]
\newtheorem{cor}[thm]{Corollary}

\newtheorem{lem}[thm]{Lemma}
\newtheorem{rmk}[thm]{Remark}
\newtheorem*{thm*}{Theorem}
\newtheorem*{con*}{Conjecture}
\newtheorem*{lem*}{Lemma}
\newtheorem{prop}[thm]{Proposition}
\newtheorem{defn}{Definition}[section]
{ \theoremstyle{remark} } 

\newcommand{\bigslant}[2]{{\raisebox{.2em}{$#1$}\left/\raisebox{-.2em}{$#2$}\right.}}

\begin{document}
\baselineskip=17pt
\hbox{}
\medskip

\title{Cayley Digraphs of Matrix Rings over Finite Fields}
\author{Ye\c{s}\.im Dem\.iro\u{g}lu Karabulut}

\email{yesim.demiroglu@rochester.edu}

\address{Department of Mathematics, University of Rochester, Rochester, NY}

\maketitle

\begin{abstract}
We use the \emph{unit-graphs} and the \emph{special unit-digraphs} on matrix rings to show that every $n \times n$ nonzero matrix over $\Bbb F_q$ can be written as a sum of two $\operatorname{SL}_n$-matrices when $n>1$. We compute the eigenvalues of these graphs in terms of Kloosterman sums and study their spectral properties; and prove that if $X$ is a subset of $\operatorname{Mat}_2 (\Bbb F_q)$ with size $|X| > \frac{2 q^3 \sqrt{q}}{q - 1}$, then $X$ contains at least two distinct matrices whose difference has determinant $\alpha$ for any $\alpha \in \Bbb F_q^{\ast}$. Using this result we also prove a sum-product type result: if $A,B,C,D \subseteq \Bbb F_q$ satisfy $\sqrt[4]{|A||B||C||D|}= \Omega (q^{0.75})$ as $q \rightarrow \infty$, then $(A - B)(C - D)$ equals all of $\Bbb F_q$. In particular, if $A$ is a subset of $\Bbb F_q$ with cardinality $|A| > \frac{3} {2} q^{\frac{3}{4}}$, then the subset $(A - A) (A - A)$ equals all of $\Bbb F_q$. We also recover a classical result: every element in any finite ring of odd order can be written as the sum of two units.

\


\

\end{abstract}

\section{Introduction and Statements of Results}

Let $R$ be a finite ring with identity, and let $U$ denote the set of units. We define the \emph{unit-graph} $G$ on $R$ to equal the directed graph (digraph) whose vertex set is $R$, for which there is a directed edge from $a$ to $b$ if and only if $b - a \in U$. This is equivalent to saying that $G$ is the Cayley digraph on $R$ associated to the subset $U$, i.e. $G = \operatorname{Cay}(R, U)$. This digraph can also be viewed as an undirected graph, since the fact that $u \in U \Longleftrightarrow -u \in U$ implies that there exists an edge from $a$ to $b$ if and only if there is also an edge from $b$ to $a$.

In this paper we first study these unit-graphs in the special cases where $R$ is a finite simple ring, or equivalently where $R$ is isomorphic to the matrix ring $\operatorname{Mat}_n(\Bbb F_q)$ for some finite field $\Bbb F_q$. It is easy to see that such graphs are regular, and we show that they are connected as well. If $n = 1$, $R \cong \Bbb F_q$, so all such graphs are trivially complete. If $n \geqslant 2$, it is well known (and also shown here) that any element of $\operatorname{Mat}_n (\Bbb F_q)$ can be written as a sum of two invertible matrices, which easily implies that the diameter of the unit-graph on $\operatorname{Mat}_n (\Bbb F_q)$ is $2$ when $n \geqslant 2$.

Among other results, we prove in this paper that the adjacency matrix of the unit-graph on $\operatorname{Mat}_n (\Bbb F_q)$ has at most $n + 1$ distinct eigenvalues, from which we deduce that the unit-graph on $\operatorname{Mat}_n (\Bbb F_q)$ in the case $n = 2$ is strongly regular for any finite field $\Bbb F_q$. In addition, we calculate the spectrum and the parameters of these strongly regular graphs for varying $q$, and show that these parameters agree with those of another family of strongly regular graphs, namely \emph{Latin square graphs}.

Along with studying these unit-digraphs on the rings $R = \operatorname{Mat}_n (\Bbb F_q)$, we also define the \emph{special unit-digraphs} on such rings $R$, by replacing the subsets $U =\operatorname{GL}_n (\Bbb F_q)$ by the subsets $U' =\operatorname{SL}_n (\Bbb F_q)$, so that these digraphs equal $\operatorname{Cay}(R, U')= \operatorname{Cay} (\operatorname{Mat}_n (\Bbb F_q), \operatorname{SL}_n (\Bbb F_q))$. We then show that every nonzero element of $\operatorname{Mat}_n (\Bbb F_q)$ for $n \geqslant 2$ can be written as a sum of two $\operatorname{SL}_n$-matrices, and hence these digraphs also are connected with diameter $2$, although in this case we show that the corresponding adjacency matrices can have a larger number of distinct eigenvalues (at most $n + q - 1$ of these). Furthermore, we compute the spectrums of these digraphs in terms of Kloosterman sums in the case of $n = 2$ and apply the spectral gap theorem to prove the following:

\begin{thm} \label{intro thm}
Let $\alpha \in \Bbb F_q ^{\ast}$. If $X,Y \subseteq \operatorname{Mat}_2 (\Bbb F_q)$ satisfies $\sqrt{|X| |Y|} > \frac{2 q^3 \sqrt{q}}{q - 1}$, then there exists some $M \in X$ and $N \in Y$ such that $M-N$ has determinant $\alpha$. In particular, if $|X| > \frac{2 q^3 \sqrt{q}}{q - 1}$, then $X$ contains at least two distinct matrices whose difference has determinant $\alpha$. Thus, if $|X| = \Omega (q^{2.5})$ as $q \rightarrow \infty$, then it contains at least two distinct matrices whose difference has determinant $\alpha$.
\end{thm}

Using this result we also prove that if $A,B,C,D \subseteq \Bbb F_q$ satisfy $\sqrt[4]{|A||B||C||D|}= \Omega (q^{0.75})$ as $q \rightarrow \infty$, then $(A - B)(C - D)$ equals all of $\Bbb F_q$. In particular, if $A$ is a subset of $\Bbb F_q$ with cardinality $|A| > \frac{3} {2} q^{\frac{3}{4}}$, then the subset $(A - A)(A - A)$ equals all of $\Bbb F_q$.

As a separate result, using Artin-Wedderburn Theory together with one of our matrix ring propositions, we show that every element in any finite ring of odd order can be written as the sum of two units\footnote{This result has been known for a while (\cite{Henriksen}), but was previously proven using different and arguably more complicated methods.}.

\section{Some Linear Algebra}

Let $\Bbb F_q$ be the finite field of order $q$ and let $\operatorname{Mat}_n(\Bbb F_q)$ be the ring of $n \times n$ matrices over $\Bbb F_q$. The general linear group $\operatorname{GL}_n(\Bbb F_q) = \{ A \in \operatorname{Mat}_n(\Bbb F_q) \mid \det(A) \neq 0 \}$ is the group of invertible matrices in $\operatorname{Mat}_n(\Bbb F_q)$ under matrix multiplication. We can easily calculate the order of this group by using the fact that a matrix is invertible if and only if its columns (or rows) are linearly independent. If $A \in \operatorname{GL}_n(\Bbb F_q)$ and $A=[v_1 \  v_2 \ \cdots \ v_n]$ for some column vectors $v_1,v_2,\cdots,v_n \in \Bbb F_q^n$, then $v_1$ can be anything but not the zero vector; $v_2$ can be anything but not a scalar multiple of $v_1$; $v_3$ can be anything but not a linear combination of $v_1$ and $v_2$ etc. This means we have $(q^n-1)$ many possibilities for $v_1$, once we pick $v_1$, $v_2$ has $(q^n-q)$ many possibilities, once we have $v_1$ and $v_2$, $v_3$ has $(q^n-q^2)$ etc. Hence, we have \begin{equation} \label{sizeGLn} 
\big|\operatorname{GL}_n(\Bbb F_q)\big|=(q^n-1)(q^n-q)\cdots (q^n-q^{n-1})=q^{n^2} \left(1-\frac{1}{q}\right) \left(1-\frac{1}{q^2}\right) \cdots \left(1-\frac{1}{q^n}\right).
\end{equation}
We use the same notation with \cite{Murphy} and define \[ \phi(n,q)=\dfrac{\big|\operatorname{GL}_n(\Bbb F_q)\big|}{\big|\operatorname{Mat}_n(\Bbb F_q)\big|}= \left(1-\frac{1}{q}\right) \left(1-\frac{1}{q^2}\right) \cdots \left(1-\frac{1}{q^n}\right).\] 
Notice that if $\phi(n,q)>\frac{1}{2}$, then every matrix in $\operatorname{Mat}_n(\Bbb F_q)$ can be written as a sum of two invertible matrices by the \emph{pigeonhole principle}: For any $A \in \operatorname{Mat}_n (\Bbb F_q)$, we have $|\operatorname{GL}_n(\Bbb F_q)| = |A - \operatorname{GL}_n(\Bbb F_q)|$. Hence if $\phi(n,q) > \frac{1}{2}$, then $\operatorname{GL}_n(\Bbb F_q) \cap (A - \operatorname{GL}_n(\Bbb F_q)) \neq \emptyset$, and the result follows.

\begin{prop} \label{qBiggerThan2}
Every element of $\operatorname{Mat}_n(\Bbb F_q)$ can be written as a sum of two invertible matrices for all $n \geqslant 1$ and all finite fields $\Bbb F_q$ as long as $q>2$.
\end{prop}

\begin{proof}
We will show that $\phi(n,q)>\frac{1}{2}$ under the assumptions and the result will follow from the above discussion. First notice that since each factor of $\phi(n,q)$ is increasing in $q$, $\phi(n,q)$ is increasing in $q$, so the general case will follow from $q=3$. Since $\phi(n,3)$ is monotonically decreasing as a function of $n$ and also bounded below by $0$, we have $\alpha:=\lim_{n \rightarrow \infty}\phi(n,3) =\left(1-\frac{1}{3}\right) \left(1-\frac{1}{3^2}\right) \cdots \left(1-\frac{1}{3^n} \right) \cdots$ exists by the monotone convergence theorem.
\begin{align*}
-\log \alpha &= \sum_{k=1}^{\infty} \frac{\left( \frac{1}{3} \right)^{k}}{k} + \sum_{k=1}^{\infty} \frac{\left( \frac{1}{3^2} \right)^{k}}{k}+\cdots\\
&=  \sum_{n=1}^{\infty}  \sum_{k=1}^{\infty} \frac{\left( \frac{1}{3^n} \right)^{k}}{k}= \sum_{k=1}^{\infty} \frac{1}{k} \sum_{n=1}^{\infty} \left( \frac{1}{3^k} \right)^{n}= \sum_{k=1}^{\infty} \frac{1}{k(3^k-1)} \\
&= \frac{1}{2}+ \frac{1}{16}+ \sum_{k=3}^{\infty} \frac{1}{k(3^k-1)}  \leqslant \frac{1}{2}+ \frac{1}{16}+ \frac{10}{9} \sum_{k=3}^{\infty} \frac{1}{k 3^k}\\
& = \frac{1}{2}+ \frac{1}{16} +\frac{10}{9} \left( \sum_{k=1}^{\infty} \frac{ \left(\frac{1}{3}\right)^k}{k}- \frac{1}{3}-\frac{1}{18} \right) =\frac{9}{16} + \frac{10}{9} \log \left( \frac{3}{2} \right) - \frac{35}{81}
\end{align*}
Since $\log \left( \frac{3}{2} \right) \approx 0.405$, we have $-\log \alpha < 0.581$ which implies $\alpha>0.5$ hence the claim follows.
\end{proof}

\begin{defn}
Let $\Bbb F$ be any field. Two $n \times n$ matrices $A, B \in \operatorname{Mat}_n(\Bbb F)$ are said to be $\operatorname{GL}_n$-\emph{equivalent} if there exists two invertible matrices $P$ and $Q$ such that $A=P B Q$. $A$ and $B$ are said to be $\operatorname{SL}_n$-\emph{equivalent} if both $P$ and $Q$ are also in $\operatorname{SL}_n (\Bbb F)$.
\end{defn}

\begin{thm} \label{equivalency}
Let $\Bbb F$ be any field and let $A,B \in \operatorname{Mat}_n(\Bbb F)$.
\vspace{-0.3cm}
\begin{itemize}
\item $A$ is $\operatorname{GL}_n$-equivalent to $B$ if and only if $\operatorname{rank} A = \operatorname{rank} B$.
\vspace{-0.3cm}
\item $A$ is $\operatorname{SL}_n$-equivalent to $B$ if and only if they have the same rank and determinant. 
\vspace{-0.3cm}
\end{itemize}
\end{thm}

\begin{proof}
Since "$\Longrightarrow$" direction is clear for both statements, we will only prove the converses. 

Let $A \in \operatorname{Mat}_n(\Bbb F)$ of rank $r$. First we want to show that by performing a finite number of modified elementary row and column operations on $A$ we can transform it into $D \in \operatorname{Mat}_n(\Bbb F)$ such that 
\begin{itemize}
\vspace{-0.3cm}
\item $D_{rr} = 
\begin{cases}
1, & \text{ if } r<n \\
\det(A) , & \text{ if } r=n
\end{cases}$
\vspace{-0.3cm}
\item $D_{ii}=1$ for $i<r$, and the rest of the entries are zero.
\end{itemize} 

To prove this claim we will use certain modified elementary operations. Adding any constant multiple of a row (column) of $A$ to another row (column) will be called an \emph{operation of type 1}; multiplying one of the rows (columns) with some nonzero number $\alpha$ and another row (column) with $\frac{1}{\alpha}$ simultaneously will be called an \emph{operation of type 2}. A result of a type 1 (or type 2) operation on $A$ can also be written as $EA$ or $AE$ (depending on if it is a row or column operation) for some $E \in \operatorname{SL}_n(\Bbb F)$. Hence, our claim is indeed that $A$ can be transformed into $D$ via the multiplication of $A$ with some $\operatorname{SL}_n$-matrices.

Once the claim is proven, that means there exists some $E_1, \dots, E_m \in \operatorname{SL}_n(\Bbb F)$ such that $ E_1 \cdots E_j A E_{j+1} \cdots E_m=D$. If $A$ and $B$ have the same rank and determinant, similarly $B$ can be transformed into the same $D$, i.e. there exists some $E'_1, \dots, E'_{m'}$ such that $E'_1 \cdots E'_{j'} B E'_{j'+1} \cdots E'_{m'}=D$. Hence we have $E_1 \cdots E_j A E_{j+1} \cdots E_m= E'_1 \cdots E'_{j'} B E'_{j'+1} \cdots E'_{m'}$. Since inverse of an $\operatorname{SL}_n$-matrix or multiplication of two $\operatorname{SL}_n$-matrices is again $\operatorname{SL}_n$, we can multiply both sides of the equation with $(E_1 \cdots E_{j})^{-1}$ from the left and with $(E_{j+1} \cdots E_{m})^{-1}$ from the right; and this proves if two matrices have the same rank and determinant, then they are $\operatorname{SL}_n$-equivalent. Hence the only thing left to show is that $A$ can be transformed into $D$ by multiplication with some matrices in $\operatorname{SL}_n(\Bbb F)$. 

If $A=[0]_{n \times n}$ (or equivalently if $\operatorname{rank} A=0$), then $r=0$ and $D=A$. Assume $A \neq [0]_{n \times n}$ from now on, so that $r>0$.

If $n=1$, then $A=[a]$ for some $a \in \Bbb F^{\ast}$. We have $r=1=n$ and $\det(A)=a$. Hence we have $D=A$. Assume $n>1$.

\textbf{Step 1:} If $(1,1)$-entry of $A$ i.e. $a_{11}$ is $1$, proceed to Step 2.

If $a_{11} =0$, then since $\operatorname{rank} A \neq 0$, there exists some  $a_{ij} \neq 0$. Add some multiple of $i^{\text{\tiny th}}$ row to the $1^{\text{\tiny st}}$ row so that $(1,j)$-entry will be $1$. Then add $j^{\text{\tiny th}}$ column to the $1^{\text{\tiny st}}$ column so that $(1,1)$-entry will be $1$. Hence this case requires at most two type 1 operations.

If $a_{11} \neq 0,1$, then multiply $A$ with $E_{a_{11}}$ from the left, where $E_{a_{11}}$ is the $n \times n$ identity matrix only with $(1,1)$-entry replaced with $\frac{1}{a_{11}}$ and $(2,2)$-entry replaced with $a_{11}$. This is an example of type 2 operation and the resultant matrix i.e. $E_{a_{11}}A$ has $(1,1)$-entry equal to $1$. 

\textbf{Step 2:} We can add the multiples of $1^{\text{\tiny st}}$ row to the other rows, and we can add the multiples of $1^{\text{\tiny st}}$ column to the other columns and eliminate all nonzero entries in the $1^{\text{\tiny st}}$ row and the $1^{\text{\tiny st}}$ column with the exception of the $1$ in the $(1,1)$-entry. Hence we transformed $A$ into $B$ such that
\[B=\begin{bmatrix}
	1   & 0  & 0    & \cdots  & 0  \\
	0   &     &       &             & \\
	0   &     &       &             & \\
\vdots &     &       &  B'        &  \\
	0   &     &       &             & 
	\end{bmatrix}_{n \times n}
	\] where $B'$ is an $(n-1) \times (n-1)$ matrix. This step requires at most $2n-2$ many type 1 operations.

\textbf{Step 3:} If $\operatorname{rank} B' = 0$ or $B'$ is a $1 \times 1$ matrix, then stop. 

Otherwise apply step 1 and 2 on $B'$ this time, and transform $B'$ into 
\[ \begin{bmatrix}
	1  & 0   & \cdots  & 0    \\
	0  &           &    & \\
\vdots &           &  B'' &  \\
	0  &           &     & 
	\end{bmatrix}_{(n-1) \times (n-1)}
	\] and check $B''$ is the zero matrix or $B''$ is a $1 \times 1$ matrix, or not. So, we can continue applying Step 1 and 2 consecutively and find $B''', B''''\dots$ etc. until eventually either one of them becomes the zero matrix or a $1 \times 1$ matrix. At the end of this process, if $r<n$ we get $r$ many ones in the diagonal and zeros everywhere else. But if $r=n$, that means $A$ is transformed into some matrix in the form of 
	\[D=\begin{bmatrix}
	1   & 0  & 0    & \cdots  & 0  \\
	0   & 1   &       &             & \\
	0   &     & \ddots      &             & \\
\vdots &     &       &  1      &  \\
	0   &     &       &             & \delta
	\end{bmatrix}_{n \times n}
	.\]
Notice that as we transformed $A$ into $D$, we only performed type 1 and type 2 operations on $A$, i.e. we multiplied $A$ with only $\operatorname{SL}_n$-matrices. This operation does not change the determinant. Hence, $\det A=\det D= \delta$ and the claim follows.

Assume $A$ and $B$ have the same rank, but not necessarily the same determinant. Then notice that when we apply the above process, if their determinant is zero, then we get the same $D$ for both of them. But if their determinants are nonzero, then it is easy to show that both $A$ and $B$ are $\operatorname{GL}_n$-equivalent to $n \times n$ identity matrix and this finishes the proof. 
\end{proof}
In the following remark, we note a small observation which will be very useful later for some of our calculations.
\begin{rmk} \label{rank}
Let $A$ be $\operatorname{GL}_n$ [resp. $\operatorname{SL}_n$]-equivalent to $B$. $A$ can be written as a sum of two $\operatorname{GL}_n$ [resp. $\operatorname{SL}_n$]-matrices if and only if $B$ can be written as a sum of two $\operatorname{GL}_n$ [resp. $\operatorname{SL}_n$]-matrices.
\end{rmk}

\begin{prop} \label{qEquals2}
Every element of $\operatorname{Mat}_n(\Bbb F_2)$ can be written as a sum of two invertible matrices when $n \geq 2$.
\end{prop}

\begin{proof}
First notice that if we can write a matrix with rank $r$ as a sum of two units (i.e. invertible matrices), then Remark~\ref{rank} combined with the previous theorem implies that every matrix with rank $r$ can be written as a sum of two units for some units. Therefore to prove the result for $n=2$ and $n=3$ cases, we write here one arbitrary element from each rank as a sum of two units:

\begin{flalign*} & \text{rank } 1:  \qquad \begin{bmatrix}
		1 & 0 \\
		0 & 0 \\
		\end{bmatrix} = \begin{bmatrix}
		0 & 1 \\
		1 & 0 \\
		\end{bmatrix} + \begin{bmatrix}
		1 & 1 \\
		1 & 0 \\
\end{bmatrix}, &\qquad \text{rank } 2: & \qquad \begin{bmatrix}
		1 & 0 \\
		0 & 1 \\
		\end{bmatrix}= \begin{bmatrix}
		1 & 1 \\
		1 & 0 \\
		\end{bmatrix}+ \begin{bmatrix}
		0 & 1 \\
		1 & 1 \\
\end{bmatrix}  \\
& & &
\\
& \text{rank } 1:  \qquad \begin{bmatrix}
		1 & 0 & 0 \\
		0 & 0 & 0 \\
		0 & 0 & 0 \\
		\end{bmatrix} = \begin{bmatrix}
		0 & 1 & 0 \\
		1 & 0 & 0 \\
		0 & 0 & 1 \\
		\end{bmatrix} + \begin{bmatrix}
		1 & 1 & 0 \\
		1 & 0 & 0 \\
		0 & 0 & 1 \\
		\end{bmatrix}, &\qquad \text{rank } 2: & \qquad \begin{bmatrix}
      	1 & 0 & 0 \\
      	0 & 1 & 0 \\
      	0 & 0 & 0 \\
      	\end{bmatrix} = \begin{bmatrix}
      	0 & 1 & 0 \\
      	0 & 0 & 1 \\
      	1 & 0 & 0 \\
      	\end{bmatrix} + \begin{bmatrix}
      	1 & 1 & 0 \\
      	0 & 1 & 1 \\
      	1 & 0 & 0 \\
      	\end{bmatrix} \\
& & &
\\
& \text{rank } 3:  \qquad \begin{bmatrix}
      	1 & 0 & 0 \\
      	0 & 1 & 0 \\
      	0 & 0 & 1 \\
      	\end{bmatrix}= \begin{bmatrix}
      	1 & 1 & 0 \\
      	0 & 0 & 1 \\
      	1 & 0 & 0 \\
      	\end{bmatrix} + \begin{bmatrix}
      	0 & 1 & 0 \\
      	0 & 1 & 1 \\
      	1 & 0 & 1 \\
      	\end{bmatrix}	& &
\end{flalign*}	
Besides the only matrix with rank zero is the zero matrix, and it can be written as a sum of the identity matrix and negative of the identity matrix. This calculation completes the result for $n=2$ and $n=3$. We use induction for $n \geqslant 4$. Let $n \geqslant 4$ be fixed. Assume every element of $\operatorname{Mat}_k(\Bbb F_2)$ can be written as a sum of two units in $\operatorname{Mat}_k(\Bbb F_2)$ for $1 < k < n$. Let $ A \in \operatorname{Mat}_n(\Bbb F_2)$.\\

\noindent \textbf{Case 1:} $A$ is invertible so $\operatorname{rank} A=n$. Since the $n \times n$ identity matrix $I_n$ and $A$ are $\operatorname{GL}_n$-equivalent, we will show that $I_n$ can be written as a sum of two units and the result will follow for $A$ by Remark~\ref{rank}. We have \[ I_n=\begin{bmatrix}
	1 & &  &  &      \\
	   & 1 &      &\text{\huge0}      & \\
	  &      & \ddots&  & \\
	 & \text{\huge0}     &      &1 & \\
	        &  &  &  & 1 
	\end{bmatrix}_{n \times n} = \begin{bmatrix}
\begin{bmatrix}
1&0\\
0&1
\end{bmatrix}_{2 \times 2}	&\text{\huge0}&\\
\text{\huge0}	&\begin{bmatrix}
	1 & &\text{\huge0}\\
	&\ddots&\\
	\text{\huge0} & & 1
	\end{bmatrix}
	\end{bmatrix}_{n \times n}.\] 
By the calculation before, we have $ \begin{bmatrix}
		1 & 0 \\
		0 &1 \end{bmatrix}=A_1 +A_2 $
for some invertible $2 \times 2$ matrices $A_1$ and $A_2$.\\

\noindent By induction hypothesis, we have $ \begin{bmatrix}
		1 & & \\
		& \ddots & \\
		&  & 1 \end{bmatrix} = A_3 +A_4$ for some invertible $(n-2) \times (n-2)$ matrices $A_3$ and $A_4$. Hence,
 \[ I_{n}= \begin{bmatrix}
 A_1 & 0 \\
 0   & A_3
 \end{bmatrix} + \begin{bmatrix}
 A_2 & 0 \\
 0   & A_4
 \end{bmatrix}.\]
 Moreover, since $\det \begin{bmatrix}
 A_1 & 0 \\
 0   & A_3
 \end{bmatrix}=\det( A_1) \det( A_3)$ and similarly for $\begin{bmatrix}
 A_2 & 0 \\
 0   & A_4
 \end{bmatrix}$, both $\begin{bmatrix}
 A_1 & 0 \\
 0   & A_3
 \end{bmatrix}$ and $\begin{bmatrix}
 A_2 & 0 \\
 0   & A_4
 \end{bmatrix}$ are invertible.
 
\noindent \textbf{Case 2:} $A$ is not invertible, then $\operatorname{rank} A = r < n$. Recall that elementary row and column operations do not change the rank of a matrix, so by performing elementary operations on $A$, we can transform $A$ into some $B$ such that all of the entries in the $n^{\text{\tiny th}}$ row and $n^{\text{\tiny th}}$ column of $B$ are zero. Then
\[ B = \begin{bmatrix}
	A_1 & 0\\
	0&1
	\end{bmatrix}+ \begin{bmatrix}
	A_2 & 0\\
	0& -1
	\end{bmatrix} \] for some $A_1,A_2 \in \operatorname{GL}_{n-1}(\Bbb F_2)$ by the induction hypothesis.
\end{proof}
The following corollary recovers a classical result, see \cite{Wolfson}, \cite{Zelinsky}.
\begin{cor} \label{everyone}
Except in the trivial case where $n = 1$ and $q = 2$, any element of $\operatorname{Mat}_n (\Bbb F_q)$ can be written as a sum of two invertible matrices.
\end{cor}
The corollary follows from Proposition~\ref{qBiggerThan2} and~\ref{qEquals2}. Furthermore, we can combine this result with classical finite ring results (Artin-Wedderburn Theory) and recover a classical result (\cite{Henriksen}):
\begin{cor} 
If $R$ is a finite ring with identity and if its order is odd, then every element of $R$ is the sum of two units.
\end{cor}
\begin{proof}
Consider $\bigslant{R}{J}$, where $J$ denotes the Jacobson radical of $R$. First notice that $\bigslant{R}{J}$ is semisimple i.e. $J \Big( \bigslant{R}{J} \Big) = 0$. Moreover, since $R$ is finite, $\bigslant{R}{J}$ is finite so $\bigslant{R}{J}$ is both left and right Artinian. Artin-Wedderburn theorem implies that $\bigslant{R}{J} \cong \operatorname{Mat}_{n_{1}}(D_1) \times \cdots \times \operatorname{Mat}_{n_{r}}(D_r)$ for some $D_1, \cdots D_r$ division rings (see \cite{Farb-Algebra}, \cite{Hungerford}). Since $\bigslant{R}{J}$ is finite, each $D_i$ has to have finitely many elements. By Wedderburn's little theorem $D_i$'s are finite fields. Therefore, we have $\bigslant{R}{J} \cong \operatorname{Mat}_{n_{1}}(\Bbb F_{q_1}) \times \cdots \times \operatorname{Mat}_{n_{r}}(\Bbb F_{q_r})$ for some finite fields $\Bbb F_{q_1}, \cdots, \Bbb F_{q_r}$. Since $R$ has odd order, $2 \nmid | \bigslant{R}{J} |$, hence none of the $\Bbb F_{q_i}$'s is the finite field of order $2$. As a result of the previous corollary, every element of $\bigslant{R}{J}$ can be written as a sum of two units and this implies every element of $R$ is the sum of two units, see \cite{Demiroglu1} for more details.
\end{proof}

\section{Special Graphs on Matrix Rings over Finite Fields}

In this section we study the unit-graph and the special unit-digraph on $\operatorname{Mat}_n(\Bbb F_q)$ and prove the rest of the results which we referred to in the introduction earlier.
\begin{prop}\label{connected}
The unit-graph on $\operatorname{Mat}_n(\Bbb F_q)$ is connected and $\vert \operatorname{GL}_n(\Bbb F_q) \vert$-regular.
\end{prop}

\begin{proof} 
Let $A \in \operatorname{Mat}_n(\Bbb F_q)$. There is an edge between $A$ and $B$ if and only if $B-A$ is an invertible matrix. That means every time we add an invertible matrix to $A$, we get a neighbor of $A$; and they are all distinct as $A+C_1=A+C_2$ implies $C_1=C_2$. So, the degree of vertex $A$ is $\vert \operatorname{GL}_n(\Bbb F_q) \vert$. To prove the first part of the claim, assume $\operatorname{Mat}_n(\Bbb F_q) \neq \Bbb F_2$ and let $A,B \in \operatorname{Mat}_n(\Bbb F_q)$. By Corollary~\ref{everyone}, we know that $B-A$ can be written as a sum of two units, i.e. $B-A=C_1+C_2$ for some $C_1, C_2 \in \operatorname{GL}_n(\Bbb F_q)$. Then, there is an edge from $A$ to $A+C_1$, and there is an edge from $A+C_1$ to $B$. This means the graph is connected.   
\end{proof}
As a side note, notice that this proof implies that the unit-graph on $\operatorname{Mat}_n(\Bbb F_q)$ has diameter at most two, but it is easy to see that it is actually precisely two when $n \geqslant 2$, and it has diameter one when $n=1$. 
\begin{prop}\label{NumberOfEigenvalues}
The unit-graph on $\operatorname{Mat}_n(\Bbb F_q)$ has at most $n+1$ distinct eigenvalues. 
\end{prop}

\begin{proof} 
Since the unit-graph on $\operatorname{Mat}_n(\Bbb F_q)$ is a Cayley digraph, we can find its eigenvectors and eigenvalues using Theorem~\ref{character theory}. Let $\tilde{\chi}$ be a character on $\operatorname{Mat}_n(\Bbb F_q)$. Then, it can be written as $\tilde{\chi}(s)=\chi(\operatorname{Tr}(As))=\chi_{A}(s)$ for some $A \in \operatorname{Mat}_n(\Bbb F_q)$ where $\operatorname{Tr}$ denotes the matrix trace and $\chi$ stands for the canonical character on $\Bbb F_q$, see \cite{Demiroglu1}.

Let $A \in \operatorname{Mat}_n(\Bbb F_q)$. Then the eigenvalue corresponding to $A$ is explicitly $\lambda_{A}= \sum_{s \in \operatorname{GL}_n(\Bbb F_q)} \chi( \operatorname{Tr}(As))$ by Theorem~\ref{character theory}. If $A, B \in \operatorname{Mat}_n(\Bbb F_q)$ are $\operatorname{GL}_n$-equivalent, i.e. if there exists two invertible matrices $P$ and $Q$ such that $A=P B Q$, then $\lambda_{A}=\lambda_{B}$. This a consequence of the following calculation.
\[ \lambda_{A} = \sum_{s \in \operatorname{GL}_n(\Bbb F_q)} \chi( \operatorname{Tr}(PBQs)) =\sum_{s' \in \operatorname{GL}_n(\Bbb F_q)} \chi( \operatorname{Tr}(PBs'))= \sum_{s' \in \operatorname{GL}_n(\Bbb F_q)} \chi( \operatorname{Tr}(Bs'P))
= \sum_{s'' \in \operatorname{GL}_n(\Bbb F_q)}\chi( \operatorname{Tr}(Bs'')) \]

Recall by Theorem~\ref{equivalency} we know that two matrices are $\operatorname{GL}_n$-equivalent if and only if their rank is the same, this implies we have at most $n+1$ distinct eigenvalues.
\end{proof}

Let $G$ be a non-empty and not complete regular graph. Recall that $G$ is called a \emph{strongly regular graph} with parameters $(n,k,a,c)$ if it is $k$-regular, every pair of distinct adjacent vertices has $a$ common neighbors and every pair of distinct nonadjacent vertices has $c$ common neighbors. We note the following well-known fact as a lemma, and its proof can be found on page 220 in \cite{Godsil}.
\begin{lem} 
A connected regular graph with exactly three distinct eigenvalues is strongly regular. 
\end{lem}
This lemma combined with Proposition~\ref {connected} and~\ref{NumberOfEigenvalues} yields the following result.
\begin{cor}
The unit-graph on $\operatorname{Mat}_2(\Bbb F_q)$ is strongly regular for any finite field $\Bbb F_q$.
\end{cor}
Since we found the unit-graph on $\operatorname{Mat}_2(\Bbb F_q)$ is strongly regular, then one wonders what the parameters are for this graph. 

\begin{thm} \label{stronglyRegular} 
The unit-graph on $\operatorname{Mat}_2(\Bbb F_q)$ is a strongly regular graph with parameters $(q^4, q^4-q^3-q^2+q, q^4-2q^3-q^2+3q, q^4-2q^3+q)$.
\end{thm} 

\begin{proof}
By Proposition~\ref{connected} we know that the unit-graph on $\operatorname{Mat}_2(\Bbb F_q)$ is $\vert \operatorname{GL}_2(\Bbb F_q) \vert$-regular. By Equation~\eqref{sizeGLn} the order of $\operatorname{GL}_2(\Bbb F_q)$ is $(q^2-1)(q^2-q)$, so the  graph is $(q^4-q^3-q^2+q)$-regular.

Now we want to count the number of common neighbors of a pair of nonadjacent vertices. We take $\begin{bmatrix}
		0 & 0 \\
		0 & 0 \\
	  \end{bmatrix}$ and $\begin{bmatrix}
		1 & 0 \\
		0 & 0 \\
	  \end{bmatrix}$, since they are nonadjacent. We assume $A=\begin{bmatrix}
		a & b \\
		c & d \\
	  \end{bmatrix}$ to satisfy $A \sim \begin{bmatrix}
		0 & 0 \\
		0 & 0 \\
	  \end{bmatrix}$ and $A \sim \begin{bmatrix}
		1 & 0 \\
		0 & 0 \\
	  \end{bmatrix}$. Then, $\det A \neq 0$ and $A - \begin{bmatrix}
1&0\\
0&0
\end{bmatrix} = \begin{bmatrix}
a-1&b\\
c&d
\end{bmatrix} $ should be a unit.

\textbf{Case 1:} Exactly two entries of $A$ are nonzero.
\vspace{-0.3cm}
\begin{itemize}
\item Let $A=\begin{bmatrix}
0&b\\
c&0
\end{bmatrix} $ where $b$ and $c$ are nonzero. Then, $A - \begin{bmatrix}
1&0\\
0&0
\end{bmatrix}=\begin{bmatrix}
	-1&b\\
	c&0
\end{bmatrix}$ should be a unit. Since $b$ has $q-1$ many different possibilities and so does $c$, in total we have $(q-1)^2$ many different possibilities for $A$.
\vspace{-0.3cm}
\item Let $A= \begin{bmatrix}
a&0\\
0&d
\end{bmatrix} $ where $a$ and $d$ are nonzero. Then, $A - \begin{bmatrix}
1&0\\
0&0
\end{bmatrix}= \begin{bmatrix}
a-1&0\\
0&d
\end{bmatrix}$ and $(a-1)d \neq 0$, which means $a \neq 1$. Since $a$ has $q-2$ and $d$ has $q-1$ many different possibilities, in total we have $(q-1)(q-2)$ many different possibilities for $A$.
\end{itemize}
\textbf{Case 2:} Exactly three entries of $A$ are nonzero.
\vspace{-0.3cm}
\begin{itemize}
	\item Let $A= \begin{bmatrix}
	0&b\\
	c&d
	\end{bmatrix} $ where $b,c,d$ are all nonzero. Also, $ \begin{bmatrix}
	-1&b\\
	c&d
	\end{bmatrix}$ should be a unit. Hence $-d-bc \neq 0$. This means $ -d \neq bc $ i.e. $ 1 \neq -bc d^{-1} $. Since $b$ has $q-1$, $c$ has $q-1$ and $d$ has $q-2$ many different possibilities, in total we have $(q-1)^2  (q-2)$ many different possibilities for $A$.
\vspace{-0.3cm}
\item Let $A=\begin{bmatrix}
		a&0\\
		c&d
	\end{bmatrix}$ where $a,c,d$ are all nonzero. Also, $ \begin{bmatrix}
		a-1&0\\
		c&d
	\end{bmatrix}$ should be a unit. Hence, $(a-1)d \neq 0$ i.e. $a \neq 1$. Since $a$ has $q-2$, $c$ has $q-1$ and $d$ has $q-1$ many different possibilities, in total we have $(q-1)^2  (q-2)$ many different possibilities for $A$.
\vspace{-0.3cm}
\item  Let $A= \begin{bmatrix}
a&b\\
0&d
\end{bmatrix} $ where $a,b,d$ are all nonzero. This case is the same with $A=\begin{bmatrix}
a&0\\
c&d
\end{bmatrix}$ case, so we have in total $ (q-1)^2 (q-2)$ many different possibilities for $A$ in this case.
\vspace{-0.3cm}
\item Let $A= \begin{bmatrix}
	a&b\\
	c&0
	\end{bmatrix} $ where $a,b,c$ are all nonzero. Here, $\begin{bmatrix}
	a-1 & b\\
	c & 0
	\end{bmatrix}$ should be a unit but since $b$ and $c$ are both nonzero, it is automatically satisfied. Hence each one of $a,b,c$ has $q-1$ many different possibilities, in total we have $(q-1)^3$ many different possibilities for $A$.
\end{itemize}
\textbf{Case 3:} None of the entries of $A$ is zero. We have $ \det \begin{bmatrix}
a&b\\
c&d
\end{bmatrix} \neq 0 $ and $ \det \begin{bmatrix}
a-1&b\\
c&d
\end{bmatrix} \neq 0$.
\vspace{-0.3cm}
\begin{itemize}
\item Assume $bc d^{-1} = -1$. We have $ad-bc \neq 0$, i.e. $a \neq bc d^{-1}$. We also have $(a-1)d-bc \neq 0 $. This means $a \neq bc d^{-1} +1$. We have $q-1$ for $b$, $q-1$ for $c$, only $1$ choice for $d$ and $q-2$ many possibilities for $a$. In total we have $(q-1)^2 (q-2)$ many possibilities.
\vspace{-0.3cm}	
\item Assume $bc d^{-1} \neq -1$. We have $q-1$ many different possibilities for $b$, $q-1$ many for $c$, $q-2$ for $d$ and $q-3$ many for $a$ (since $a \neq 0$, $a \neq bc d^{-1}$ and $a \neq bc d^{-1}+1$). We have in total $(q-1)^2 (q-2) (q-3)$ many possibilities for $A$. 
\end{itemize}
As a result, case 1 gives us $2q^2-5q+3$ many different possibilities for $A$, case 2 gives $4q^3-15q^2+18q-7$ and 
case 3 gives $q^4-6q^3+13q^2-12q+4$ which sums to $q^4-2q^3+q$.
  
Now we want to count the number of common neighbors of a pair of adjacent vertices. We pick $\begin{bmatrix}
		1 & 0 \\
		0 & 1 \\
	  \end{bmatrix}$ and $\begin{bmatrix}
		0 & 0 \\
		0 & 0 \\
	  \end{bmatrix}$, since they are adjacent. We assume $A=\begin{bmatrix}
		a & b \\
		c & d \\
	  \end{bmatrix}$ satisfy $A \sim \begin{bmatrix}
		1 & 0 \\
		0 & 1 \\
	  \end{bmatrix}$ and $A \sim \begin{bmatrix}
		0 & 0 \\
		0 & 0 \\
	  \end{bmatrix}$. Then, $\det A \neq 0$ and $A - \begin{bmatrix}
1&0\\
0&1
\end{bmatrix} = \begin{bmatrix}
a-1&b\\
c&d-1
\end{bmatrix}$ should be a unit.

\textbf{Case 1:} Exactly two entries of $A$ are nonzero.
\vspace{-0.3cm}
\begin{itemize}
\item Let $A=\begin{bmatrix}
0&b\\
c&0
\end{bmatrix} $ where $b$ and $c$ are nonzero. We have $A - \begin{bmatrix}
1&0\\
0&1
\end{bmatrix}=\begin{bmatrix}
	-1&b\\
	c&-1
\end{bmatrix}$ and $1-bc \neq 0$ i.e. $b \neq c^{-1}$. Hence $b$ has $q-1$ and $c$ has $q-2$ many different possibilities, in total we have $(q-1)(q-2)$ many different possibilities for $A$.
\vspace{-0.3cm}
\item Let $A= \begin{bmatrix}
a&0\\
0&d
\end{bmatrix} $ where $a$ and $d$ are nonzero. We have $A - \begin{bmatrix}
1&0\\
0&1
\end{bmatrix}= \begin{bmatrix}
a-1&0\\
0&d-1
\end{bmatrix}$ and $(a-1)(d-1) \neq 0$ which means both $a \neq 1$ and $d \neq 1$. Since $a$ has $q-2$ many possibilities and so does $d$, in total we have $(q-2)^2$ many different possibilities for $A$.
\end{itemize}
\textbf{Case 2:} Exactly three entries of $A$ are nonzero.
\vspace{-0.3cm}
\begin{itemize}
	\item Let $A= \begin{bmatrix}
	0&b\\
	c&d
	\end{bmatrix} $ where $b,c,d$ are all nonzero. We also have that $ \begin{bmatrix}
	-1&b\\
	c&d-1
	\end{bmatrix}$ is a unit, so $1-d \neq bc$.
\vspace{-0.3cm}
\begin{itemize}
\item[$\square$] If $d=1$, $b$ and $c$ can be anything. So we have $q-1$ many for $b$ and $q-1$ many for $c$, in total it is $(q-1)^2$.
\vspace{-0.3cm}
\item[$\square$] If $d \neq 1$, then we have $q-2$ many for $d$, $q-1$ many for $b$ and $q-2$ many for $c$, so in total it is $(q-1)(q-2)^2$.
\end{itemize}
\vspace{-0.3cm}
\item Let $A=\begin{bmatrix}
a&0\\
c&d
\end{bmatrix} $. We have that $ \begin{bmatrix}
a-1&0\\
c&d-1
\end{bmatrix} $ is a unit and $(a-1)(d-1) \neq 0$ i.e. $a \neq 1$, $d \neq 1$. Since $a$ has $q-2$, $c$ has $q-1$ and $d$ has $q-2$ many different possibilities, in total we have $(q-2)^2 (q-1)$ many different possibilities for $A$.
\end{itemize}	
Notice that $\begin{bmatrix}
0&b\\
c&d
\end{bmatrix}$ case and $\begin{bmatrix}
a&b\\
c&0
\end{bmatrix}$ case gives us the same count. Similarly, $\begin{bmatrix}
a&0\\
c&d
\end{bmatrix}$ and $\begin{bmatrix}
a&b\\
0&d
\end{bmatrix}$ gives the same count. Therefore, in case 2 we have $4q^3-18q^2+28q-14$ many possibilities in total.
	
\textbf{Case 3:} None of the entries of $A$ is zero. We have $ \det \begin{bmatrix}
a&b\\
c&d
\end{bmatrix} \neq 0 $ and $ \det \begin{bmatrix}
a-1&b\\
c&d-1
\end{bmatrix} \neq 0$.
\vspace{-0.3cm}
\begin{itemize}
\item $a=1$, $d \neq 1$. Since $d$ has $q-2$, $b$ has $q-1$ and $c$ has $q-2$ many different possibilities, in total we have $(q-2)^2 (q-1)$ many different possibilities for $A$.
\vspace{-0.3cm}
\item $a \neq 1$, $d=1$. This gives the same count with $a=1$, $d \neq 1$ case.
\vspace{-0.3cm}
\item $a=1$, $d=1$.
$b$ has $q-1$ and $c$ has $q-2$, we have $(q-1)(q-2)$ many different possibilities for $A$.
\vspace{-0.3cm}
\item $a \neq 1$, $d \neq 1$.
\vspace{-0.3cm}
	\begin{itemize}
		\item[$\blacksquare$] $a+d=1$. Notice $a+d=1 \iff ad+1-a-d=ad \iff (a-1)(d-1)=ad$. Since $a$ has $q-2$, $d$ has $1$, $b$ has $q-1$ and $c$ has $q-2$ many different possibilities, in total we have $(q-1)(q-2)^2$ many different possibilities for $A$.
		\vspace{-0.3cm}
		\item[$\blacksquare$] $a+d \neq 1$. Since $a$ has $q-2$, $d$ has $q-3$, $b$ has $q-1$ and $c$ has $q-3$ many different possibilities, in total we have $(q-3)^2 (q-2) (q-1)$ many different possibilities for $A$.
\end{itemize}
\end{itemize}
As a result, case 1 gives us $2q^2-7q+6$ many different possibilities for $A$, case 2 gives $4q^3-18q^2+28q-14$ and case 3 gives $q^4-6q^3+15q^2-18q+8$ which sums to $q^4-2q^3-q^2+3q$.
\end{proof}

Note that the last theorem shows us that the parameters of the unit-graph on $\operatorname{Mat}_2(\Bbb F_q)$ agrees with the parameters of another family of strongly regular graphs, \emph{Latin square graphs}. It is a well-known fact that an \emph{orthogonal array} $OA(k,n)$ is equivalent to a set of $k-2$ mutually orthogonal Latin squares. Moreover, the graph defined by $OA(k,n)$ is strongly regular with parameters \[(n^2, (n-1)k, n-2+(k-1)(k-2),k(k-1)),\] by Theorem 10.4.2 in \cite{Godsil}. Hence, if we let $n=q^2$ and $k=q^2-q$ then the parameters of $OA(k,n)$ becomes the same with the parameters of the unit-graph on $\operatorname{Mat}_2(\Bbb F_q)$. In general, a strongly regular graph is not determined by its parameters, and one can wonder that in this case if the unit-graph on $\operatorname{Mat}_2(\Bbb F_q)$ and $OA(k,n)$ are isomorphic or not. Unfortunately, we do not have a definite answer to that in this paper.

The following lemma is again a well-known fact, we refer the reader to \cite{Godsil} for the proof.
\begin{lem}
Let $\Bbb A$ be the adjacency matrix of the $(n,k,a,c)$ strongly regular graph $G$. Let $\Delta=(a-c)^2+4(k-c)$. Then,  
\begin{itemize}
\item the largest eigenvalue of $\Bbb A$ is $\lambda_1= k$ with multiplicity $1$;
\item the other eigenvalues of $\Bbb A$ are $\lambda_2= \frac{(a-c)+\sqrt{\Delta}}{2}$ with multiplicity $m_2$ and $\lambda_3= \frac{(a-c)-\sqrt{\Delta}}{2}$ with multiplicity $m_3$ where 
\[ m_2= \frac{1}{2} \left( (n-1)-\frac{2k+(n-1)(a-c)}{\sqrt{\Delta}}\right) \text{ and } m_3= \frac{1}{2}\left( (n-1)+\frac{2k+(n-1)(a-c)}{\sqrt{\Delta}}\right).\]
\end{itemize}
\end{lem} 

\begin{cor} \label{StrongEigen} 
The spectrum of the unit-graph on $\operatorname{Mat}_2(\Bbb F_q)$ is $\{(q^4-q^3-q^2+q,1), (q,q^4-q^3-q^2+q), (q-q^2,q^3+q^2-q-1) \}$ where the first entries of the ordered pairs denote the eigenvalues and the second entries are the corresponding multiplicities. 
\end{cor} 

\begin{proof}
Since we know the parameters of the unit-graph on $\operatorname{Mat}_2(\Bbb F_q)$ from Theorem~\ref{stronglyRegular}, we can plug the parameters in the formulae in the previous lemma, and the result follows. 
\end{proof}

\begin{prop}
The formulae
\[\sum_{ s \in \operatorname{GL}_2(\Bbb F_q)} \chi( s_{11})=q-q^2  \quad \text{and} \quad \sum_{ s \in \operatorname{GL}_2(\Bbb F_q)} \chi( s_{11}+ s_{22})= \sum_{ s \in \operatorname{GL}_2(\Bbb F_q)} \chi( s_{11}) \chi (s_{22})= q\] hold for any finite field $\Bbb F_q$ and for any non-trivial character on $\Bbb F_q$.
\end{prop}

\begin{proof}
Consider the unit-graph on $\operatorname{Mat}_2(\Bbb F_q)$. Let $A = \begin{bmatrix}
1&0\\
0&0
\end{bmatrix}$ and $B = \begin{bmatrix}
1&0\\
0&1
\end{bmatrix}$. We have \[ \lambda_{A} = \sum_{s \in \operatorname{GL}_2(\Bbb F_q)} \chi( \operatorname{Tr}(As))= \sum_{ s \in \operatorname{GL}_2(\Bbb F_q)} \chi( s_{11})  \quad \text{and} \quad \lambda_{B} = \sum_{s \in \operatorname{GL}_2(\Bbb F_q)} \chi( \operatorname{Tr}(Bs))= \sum_{ s \in \operatorname{GL}_2(\Bbb F_q)} \chi( s_{11}+s_{22}).\]
Since matrices of same rank correspond to same eigenvalue, we know the multiplicity of $\lambda_{A}$ denoted by $m_A \geqslant$ the number of rank $1$ matrices, and the multiplicity of $\lambda_{B}$ denoted by $m_B \geqslant | \operatorname{GL}_2(\Bbb F_q) |$. But we also have $q^4=m_A+m_B+1$. This forces $m_A$ to be the number of rank $1$ matrices and $m_B$ to be $| \operatorname{GL}_2(\Bbb F_q) |$. Hence, $m_A=q^4-| \operatorname{GL}_2(\Bbb F_q) |-1= q^3+q^2-q-1$ and $m_B= q^4-q^3-q^2+q$. By Corollary~\ref{StrongEigen} we know that $q-q^2$ is an eigenvalue of this graph with multiplicity $q^3+q^2-q-1$ and $q$ is an eigenvalue with multiplicity $q^4-q^3-q^2+q$. Hence, $ \lambda_{A}$ should be $q-q^2$ and  $\lambda_{B}$ should be $q$.
\end{proof}

Since we know every element of $\operatorname{Mat}_n(\Bbb F_q)$ can be written as a sum of two invertible matrices as long as $\operatorname{Mat}_n(\Bbb F_q) \neq \Bbb F_2$, one may wonder similarly if every element of the matrix ring can be written as a sum of two $\operatorname{SL}_n$-matrices under some mild hypothesis or not. In the following proposition, we first consider a special case of this question, namely when $n=2$.

\begin{prop} \label{SpecialUnits}
Every element of $\operatorname{Mat}_2(\Bbb F_q)$ can be written as a sum of two $\operatorname{SL}_2$-matrices. 
\end{prop}

\begin{proof} 
By Remark~\ref{rank} we know that if we can write a matrix as a sum of two $\operatorname{SL}_2$-matrices, then every $2 \times 2$ matrix with the same rank and determinant can be written as a sum of two $\operatorname{SL}_2$-matrices. We will show that the zero matrix, a matrix with rank one and a matrix with determinant $\alpha$ for some $\alpha \in \Bbb F_q^{\ast}$, in particular $\begin{bmatrix}
		1 & 0 \\
		0 & 0 \\
		\end{bmatrix}$ and $\begin{bmatrix}
		1 & 0 \\
		0 & \alpha \\
		\end{bmatrix}$, can be written as a sum of two $\operatorname{SL}_2$-matrices and the claim will follow from it. 
		\[ \begin{bmatrix}
		1 & 0 \\
		0 & 1 \\
	  \end{bmatrix} + \begin{bmatrix}  
	 	-1 & 0 \\
		0 & -1 \\
	    \end{bmatrix} = \begin{bmatrix}
		0 & 0 \\
		0 & 0 \\
		\end{bmatrix}  \]

\[ \begin{bmatrix}
		0 & -1 \\
		1 & 0 \\
	  \end{bmatrix} + \begin{bmatrix}  
	 	1 & 1 \\
		-1 & 0 \\
	    \end{bmatrix} = \begin{bmatrix}
		1 & 0 \\
		0 & 0 \\
		\end{bmatrix}  \]
	
Let $\alpha$ be any nonzero element of $\Bbb F_q$. Then, $\alpha^{-1}$ exists and we have:

\[ \begin{bmatrix}
		0 & \alpha^{-1} \\
		-\alpha & \alpha \\
	  \end{bmatrix} + \begin{bmatrix}  
	 	1 & -\alpha^{-1} \\
		\alpha & 0 \\
	    \end{bmatrix} = \begin{bmatrix}
		1 & 0 \\
		0 & \alpha \\
		\end{bmatrix}.\]
		
\end{proof}

\begin{thm} \label{SpecialUnitsGeneral}
Let $n \geqslant 2$. Every nonzero element of $\operatorname{Mat}_n(\Bbb F_q)$ can be written as a sum of two $\operatorname{SL}_n$-matrices. If $n$ is even or if $\operatorname{char}(\Bbb F_q)=2$, then the zero matrix can be written as a sum of two $\operatorname{SL}_n$-matrices; otherwise it requires three $\operatorname{SL}_n$-matrices.
\end{thm}

\begin{proof}
First we consider the zero matrix. If $\det(A) =1= \det(-A)$, then the zero matrix can be written as $A+(-A)$ and we are done. But that happens only when $\det(A) = (-1)^n \det(A) = \det(-A)$ is satisfied, which is the case either when $n$ is even or when we work in characteristic $2$. Hence, the zero matrix can be written as $A+(-A)$ when $n$ is even or $\operatorname{char}(\Bbb F_q)=2$. Otherwise, we can write the zero matrix as $A+(-A)$ for some $A \in \operatorname{SL}_n(\Bbb F_q)$, and once we prove every nonzero matrix can be written as a sum of two $\operatorname{SL}_n$-matrices, this implies the zero matrix can be written as a sum of three $\operatorname{SL}_n$-matrices.

Now, we focus on the nonzero matrices. We obtained the result for $n=2$ in the previous proposition. We will prove the claim for $n=3$ and then do induction on $n$. We again note that by Remark~\ref{rank} if we can write a $n \times n$ matrix as a sum of two $\operatorname{SL}_n$-matrices, then every $n \times n$ matrix with the same rank and determinant can be written as a sum of two $\operatorname{SL}_n$-matrices. Let $\alpha \in \Bbb F_q^{\ast}$. If the characteristic of $\Bbb F_q$ is $2$, then we have

\begin{flalign*} 
& \text{rank } 1: \ \begin{bmatrix}
		1 & 0 & 0 \\
		0 & 0 & 0 \\
		0 & 0 & 0 \\
		\end{bmatrix} = \begin{bmatrix}
		1 & 1 & 0 \\
		1 & 0 & 0 \\
		0 & 0 & 1 \\
		\end{bmatrix} + \begin{bmatrix}
		0 & 1 & 0 \\
		1 & 0 & 0 \\
		0 & 0 & 1 \\
		\end{bmatrix} & \text{ rank } 2: \ &  \begin{bmatrix}
      	1 & 0 & 0 \\
      	0 & 1 & 0 \\
      	0 & 0 & 0 \\
      	\end{bmatrix} = \begin{bmatrix}
      	1 & 1 & 0 \\
      	0 & 1 & 1 \\
      	1 & 0 & 0 \\
      	\end{bmatrix} + \begin{bmatrix}
      	0 & 1 & 0 \\
      	0 & 0 & 1 \\
      	1 & 0 & 0 \\
      	\end{bmatrix} \\
& & &
\\
& \text{rank } 3: \ \begin{bmatrix}
      	1 & 0 & 0 \\
      	0 & 1 & 0 \\
      	0 & 0 & \alpha \\
      	\end{bmatrix}= \begin{bmatrix}
      	1 & 1 & 0 \\
      	0 & 0 & 1 \\
      	1 & 0 & 0 \\
      	\end{bmatrix} + \begin{bmatrix}
      	0 & 1 & 0 \\
      	0 & 1 & 1 \\
      	1 & 0 & \alpha \\
      	\end{bmatrix}	.& &
\end{flalign*}	

This finishes the proof for characteristic $2$ case when $n=3$. If $\operatorname{char}(\Bbb F_q) \neq 2$, then we have

\begin{flalign*} 
& \text{rank } 1: \ \begin{bmatrix}
		2 & 0 & 0 \\
		0 & 0 & 0 \\
		0 & 0 & 0 \\
		\end{bmatrix} = \begin{bmatrix}
		1 & 0 & 0 \\
		0 & 0 & -1 \\
		0 & 1 & 0 \\
		\end{bmatrix} + \begin{bmatrix}
		1 & 0 & 0 \\
		0 & 0 & 1 \\
		0 & -1 & 0 \\
		\end{bmatrix} & \text{ rank } 2: \ &  \begin{bmatrix}
      	2 & 0 & 0 \\
      	0 & 1 & 0 \\
      	0 & 0 & 0 \\
      	\end{bmatrix} = \begin{bmatrix}
      	1 & 0  & 0 \\
      	0 & 0 & -1 \\
      	0 & 1 & 0 \\
      	\end{bmatrix} + \begin{bmatrix}
      	1 & 0 & 0 \\
      	0 & 1 & 1 \\
      	0 & -1& 0 \\
      	\end{bmatrix} \\
& & &
\\
& \text{rank } 3: \ \begin{bmatrix}
      	2 & 0 & 0 \\
      	0 & 1 & 0 \\
      	0 & 0 & \frac{\alpha}{2} \\
      	\end{bmatrix}= \begin{bmatrix}
      	1 & 0 & 0 \\
      	0 & 0 & \alpha^{-1} \\
      	0 & -\alpha & \frac{\alpha}{2} \\
      	\end{bmatrix} + \begin{bmatrix}
      	1 & 0 & 0 \\
      	0 & 1 & -\alpha^{-1} \\
      	0 & \alpha & 0 \\
      	\end{bmatrix}	. & &
\end{flalign*}	

This completes the proof for $n=3$ case.

Let $n \geqslant 4$. Assume the claim holds for all $2 \leqslant k < n$. Let $A \in \operatorname{Mat}_n(\Bbb F_q)$ of rank $r$. If $A$ is the zero matrix, we know how to handle it from the previous discussion. Hence, assume $A$ is not the zero matrix. By the proof of Theorem~\ref{equivalency} we know that $A$ is $\operatorname{SL}_n$-equivalent to $D \in \operatorname{Mat}_n(\Bbb F_q)$ such that 
\begin{itemize}
\vspace{-0.3cm}
\item $D_{rr} = 
\begin{cases}
1, & \text{ if } r<n \\
\det(A) , & \text{ if } r=n
\end{cases}$
\vspace{-0.3cm}
\item $D_{ii}=1$ for $i<r$, and the rest of the entries are zero.
\end{itemize} 
Since $n \geqslant 4$, we have $n-2 \geqslant 2$. Also, since $A$ is not the zero matrix, we know $D_{11}=1$. By the induction hypothesis there exist some matrices $A_i$'s with determinant $1$ such that \[ D= \begin{bmatrix}
 A_1 & 0 \\
 0   & A_3
 \end{bmatrix} + \begin{bmatrix}
 A_2 & 0 \\
 0   & A_4
\end{bmatrix} \] where $A_1$, $A_2$ are both $(n-2) \times (n-2)$ and $A_3$, $A_4$ are both $2 \times 2$. Hence, the claim follows.
\end{proof}
In the following corollary, we state the basic properties of the special unit-digraph on $\operatorname{Mat}_n(\Bbb F_q)$. 
\begin{cor}
Let $n \geqslant 2$. The special unit-digraph on $\operatorname{Mat}_n(\Bbb F_q)$ is connected and its diameter is $2$. It can be regarded as a simple (undirected) regular graph if and only if $n$ is even or $\operatorname{char}(\Bbb F_q)=2$. The adjacency matrix of the digraph has at most $n+q-1$ distinct eigenvalues. 
\end{cor}
\begin{proof}
As we explained in the proof of previous theorem, $1=\det(A) = \det(-A)$, i.e. $\operatorname{SL}_n(\Bbb F_q)$ is a symmetric set if and only if $n$ is even or $\operatorname{char}(\Bbb F_q)=2$. Hence if $n$ is even or if $\operatorname{char}(\Bbb F_q)=2$, then the digraph can be considered as a graph. Connectedness, regularity and the diameter proof is very similar to the proof of Proposition~\ref{connected} and follows from Theorem~\ref{SpecialUnitsGeneral}. To deduce the number of distinct eigenvalues, let $A \in \operatorname{Mat}_n(\Bbb F_q)$. Then recall the eigenvalue corresponding to $A$ is explicitly $\lambda_{A}= \sum_{s \in \operatorname{SL}_n(\Bbb F_q)} \chi( \operatorname{Tr}(As))$, and also note that if $A, B \in \operatorname{Mat}_n(\Bbb F_q)$ are $\operatorname{SL}_n$-equivalent then $\lambda_{A}=\lambda_{B}$ follows from a very similar calculation in the proof of Proposition~\ref{NumberOfEigenvalues}.
\end{proof}
Earlier in that section we proved that the unit-graph on $\operatorname{Mat}_2(\Bbb F_q)$ has at most three distinct eigenvalues and this implied that it is a strongly regular graph. Unfortunately, in general it is not the case for the special unit-graph on $\operatorname{Mat}_2(\Bbb F_q)$ except a few cases, but we can still compute its spectrum explicitly using the following theorem:
\begin{thm} \label{spect special}
The spectrum of the special unit-graph on $\operatorname{Mat}_2(\Bbb F_q)$ consists of $\lambda=q^3-q$ with multiplicity $1$, $\mu_0=-q$ with multiplicity $q^3+q^2-q-1$ and $\mu_{\delta}=q \sum_{\alpha \in \Bbb F_q^{\ast}} \chi(\alpha+ \delta \alpha^{-1})$ with multiplicity $q^3-q$ for any $\delta \in \Bbb F_q^{\ast}$. 
\end{thm}
\begin{proof}
The special unit-graph on $\operatorname{Mat}_2(\Bbb F_q)$ is a simple regular graph, hence the largest eigenvalue should be the regularity i.e. $|\operatorname{SL}_2 (\Bbb F_q)|$ with multiplicity $1$ by the Perron-Frobenius theorem. Also from a different perspective, recall that we know \[ \lambda_{A} = \sum_{s \in \operatorname{SL}_2(\Bbb F_q)} \chi( \operatorname{Tr}(As)). \] If $A$ is the zero matrix, $\lambda_{A} = \sum_{s \in \operatorname{SL}_2(\Bbb F_q)} \chi(0)= |\operatorname{SL}_2 (\Bbb F_q) | = \frac{|\operatorname{GL}_2 (\Bbb F_q) |}{q-1}= q^3-q$. 

Next we want to show that $\mu_{0}$ is the eigenvalue corresponding to rank $1$ matrices, and $\mu_{\delta}$ is the eigenvalue corresponding to matrices of determinant $\delta$. To ease our calculations to follow up, first we sort out $\operatorname{SL}_2$-matrices. Let $s \in \operatorname{SL}_2(\Bbb F_q)$, then $s$ should be in one of the following forms. 

\textbf{Case 1:} Exactly two entries of $s$ are nonzero. Then, $s$ should be either in the form of $\begin{bmatrix}
0& \alpha\\
- \alpha^{-1} &0
\end{bmatrix}$ or $ \begin{bmatrix}
\alpha&0\\
0&\alpha^{-1}
\end{bmatrix} $ for some $\alpha \in \Bbb F_q^{\ast}$.

\textbf{Case 2:} Exactly three entries of $s$ are nonzero. Then, $s$ should be in the form of $\begin{bmatrix}
	0& \alpha\\
	 -\alpha^{-1} & \beta
	\end{bmatrix}$, $\begin{bmatrix}
		 \alpha &0\\
		\beta & \alpha^{-1} 
	\end{bmatrix}$, $\begin{bmatrix}
\alpha & \beta \\
0& \alpha^{-1} 
\end{bmatrix}$ or $\begin{bmatrix}
	\beta & \alpha \\
	-\alpha^{-1} &0
	\end{bmatrix}$ for some $\alpha,\beta \in \Bbb F_q^{\ast}$.
	
\textbf{Case 3:} None of the entries of $s$ is zero. Then, $s$ should be in the form of $\begin{bmatrix}
\alpha & \gamma \\
c  & \beta
\end{bmatrix}$ where $c=\gamma^{-1} (\alpha \beta -1)$ for some $\alpha,\gamma \in \Bbb F_q^{\ast}$ and $\beta \in \Bbb F_q^{\ast}  \setminus \{ \alpha^{-1} \}$.

Let $A = \begin{bmatrix}
1&0\\
0&0
\end{bmatrix}$. Then we have 
\begin{equation} \label{SL-eigen2}
\lambda_{A} = \sum_{s \in \operatorname{SL}_2(\Bbb F_q)} \chi( \operatorname{Tr}(As))= \sum_{ s \in \operatorname{SL}_2(\Bbb F_q)} \chi( s_{11})=  \sum_{ s \in \text{Case 1}} \chi( s_{11})+  \sum_{ s \in \text{Case 2}} \chi( s_{11})+ \sum_{ s \in \text{Case 3}} \chi( s_{11}).
\end{equation}
We also have
\begin{flalign*} 
\sum_{ s \in \text{Case 1}} \chi( s_{11})&= (q-1)\chi(0)+  \sum_{ \alpha \in \Bbb F_q^{\ast}} \chi(\alpha)  \\  
\sum_{ s \in \text{Case 2}} \chi( s_{11})&=  (q-1)^2 \chi(0)  + 3 \left[ \sum_{ \alpha \in \Bbb F_q^{\ast}} \chi(\alpha) \right] ( q-1 )\\
\sum_{ s \in \text{Case 3}} \chi( s_{11})&=  \left[ \sum_{ \alpha \in \Bbb F_q^{\ast}} \chi(\alpha) \right] ( q-1) (q-2).
\end{flalign*} 
We can plug the last three equations in Equation~\eqref{SL-eigen2}, and use the orthogonality relations in character theory (in particular $\sum_{ \alpha \in \Bbb F_q} \chi(\alpha)=0$) and we get $\lambda_{A}=-q$. The multiplicity is the number of rank $1$ matrices, hence it is $q^3+q^2-q-1$.

Let $B = \begin{bmatrix}
1&0\\
0&\delta
\end{bmatrix}$. Then we have 
\begin{equation} \label{SL-eigen3}
\lambda_{B} = \sum_{s \in \operatorname{SL}_2(\Bbb F_q)} \chi( \operatorname{Tr}(Bs))= \sum_{ s \in \operatorname{SL}_2(\Bbb F_q)} \chi( s_{11}+\delta s_{22}).
\end{equation}
We also have
\begin{flalign*} 
\sum_{ s \in \text{Case 1}} \chi(  s_{11}+\delta s_{22})&= (q-1)\chi(0)+  \sum_{ \alpha \in \Bbb F_q^{\ast}} \chi(\alpha+\delta \alpha^{-1}) &\\
\sum_{ s \in \text{Case 2}} \chi(s_{11}+\delta s_{22})&= \left[ \sum_{ \alpha \in \Bbb F_q^{\ast}} \chi(\delta \alpha) \right] ( q-1 ) + (2q-2) \sum_{ \alpha \in \Bbb F_q^{\ast}} \chi(\alpha+ \delta \alpha^{-1}) + ( q-1 ) \left[ \sum_{ \alpha \in \Bbb F_q^{\ast}} \chi(\alpha) \right]&\\
\sum_{ s \in \text{Case 3}} \chi(s_{11}+\delta s_{22}) &= \left[ \sum_{ \alpha \in \Bbb F_q^{\ast}} \left( \sum_{ \beta \in \Bbb F_q^{\ast} \setminus \{ \alpha^{-1} \} } \chi ( \alpha + \delta \beta ) \right) \right] ( q-1) = \left[ \sum_{ \alpha \in \Bbb F_q^{\ast}} \left (  - \chi(\alpha)- \chi(\alpha) \chi(\delta \alpha^{-1}) \right) \right] ( q-1).&
\end{flalign*}	

We can plug the last three equations in Equation~\eqref{SL-eigen3} and we get $\lambda_{B}=q \sum_{\alpha \in \Bbb F_q^{\ast}}\chi(\alpha)\chi(\delta \alpha^{-1})$, and the multiplicity follows from the number of matrices of determinant $\delta$. 
\end{proof}
 
\begin{cor}
The equation \[ \sum_{ s \in \operatorname{SL}_2(\Bbb F_q)} \chi( s_{11}+ \delta s_{22})= \sum_{ s \in \operatorname{SL}_2(\Bbb F_q)} \chi( s_{11}) \chi (\delta s_{22})= q \sum_{\alpha \in \Bbb F_q^{\ast}} \chi \left( \alpha + \frac{\delta}{\alpha} \right) \] holds for any $\delta \in \Bbb F_q$ and any non-trivial character $\chi$ on $\Bbb F_q$. In particular $\sum_{ s \in \operatorname{SL}_2(\Bbb F_q)}\chi( s_{11})= -q$. 
\end{cor}
If we calculate the spectrum of the special unit-graph on $\operatorname{Mat}_2(\Bbb F_q)$ for some small $q$ values using Theorem~\ref{spect special}, we see that this graph is strongly regular for $q=2,3$ and $4$ since we have exactly three distinct eigenvalues in each case. However, it is not strongly regular when $q=5$. 

Similar to the definition of special unit-digraph, which is the Cayley digraph on $\operatorname{Mat}_n (\Bbb F_q)$ associated to the set of determinant $1$ matrices, one can pick an  $\alpha \in \Bbb F_q^{\ast}$ and define a new Cayley digraph $G_{\alpha}$ using the set of matrices with determinant $\alpha$. That is equivalent to saying that $G_{\alpha}$ is a Cayley digraph with the vertex set $\operatorname{Mat}_n(\Bbb F_q)$ in which there is a directed edge from $A$ to $B$ if and only if $\det (B-A) =\alpha $. It is easy to show that $G_{\alpha}$ is isomorphic to $G_{1}$ for any $\alpha \in \Bbb F_q^{\ast}$. Therefore, the spectrum of $G_{\alpha}$ is the same with the spectrum of $G_{1}$. Moreover, since each nonzero element gives us a new digraph, we get $q-1$ isomorphic digraphs. This gives us an edge-partition of the unit-graph on $\operatorname{Mat}_n(\Bbb F_q)$. Furthermore $\mu_{\delta}=q \sum_{\alpha \in \Bbb F_q^{\ast}} \chi(\alpha+\delta \alpha^{-1})$ in the previous theorem is exactly $q$ times the (classical) Kloosterman sum over $\Bbb F_q$. It is a well-known fact that this sum is bounded by the square root law i.e. \[ \left \| \sum_{\alpha \in \Bbb F_q^{\ast}} \chi\left( \alpha + \frac{\delta}{\alpha} \right) \right \| \leqslant 2 \sqrt{q}\] for any $\delta \in \Bbb F_q^{\ast}$, see e.g. \cite{Weil}. These ideas yield Theorem~\ref{intro thm}:

\begin{proof} (proof of Theorem~\ref{intro thm}.) Let $G_{\alpha}$ denote the Cayley digraph on $\operatorname{Mat}_2 (\Bbb F_q)$ associated to the set of determinant $\alpha$ matrices. We apply Theorem~\ref{SGT} on $G_{\alpha}$. We have \[ n_{\ast}= \frac{n}{|S|} \left( \max_{2 \leqslant i \leqslant n} \Big\| \sum_{s \in S} \chi_{i}(s) \Big\| \right) = \frac{2q^5 \sqrt{q}}{q^3-q}< \frac{2q^3 \sqrt{q}}{q-1}.\] Hence if $\dfrac{2q^3 \sqrt{q}}{q-1} < \sqrt{|X| |Y|}$, then the result follows from Theorem~\ref{SGT}.
\end{proof}

The last result shows that if $|X| = \Omega (q^{2.5})$ as $q \rightarrow \infty$, then $X$ contains at least two distinct matrices whose difference has determinant $\alpha$, the next non-example shows that $|X| = \Omega (q^{2})$ would not work.

\textbf{Non-example.} Let $X=\{A \in \operatorname{Mat}_2(\Bbb F_q) \mid a_{21}=0 \text{ and } a_{22}=0\}$. Then, notice that $|X|=q^2 < \frac{2q^3 \sqrt{q}}{q-1}$ and for any $A,B \in X$ we have $\det(A-B)=0$. 

Theorem~\ref{intro thm} surprisingly yields some results related to sum-product problem: 

\begin{cor}
If $A,B,C,D$ are subsets of $\Bbb F_q$ satisfying $\sqrt[4]{|A||B||C||D|} > \frac{\sqrt{2}q^{\frac{5}{4}}}{\sqrt{q-1}}$, then the subset $(A - B) (C -D)$ equals all of $\Bbb F_q$.
\end{cor}

\begin{proof}
Consider the set of matrices $X=\{M \in \operatorname{Mat}_2(\Bbb F_q) \mid m_{11} \in A, m_{22} \in C \text{ and } m_{21}=0\}$ and $Y=\{M \in \operatorname{Mat}_2(\Bbb F_q) \mid m_{11} \in B, m_{22} \in D \text{ and } m_{21}=0\}$. We have $|X|=q|A||C|$ and $|Y|=q|B||D|$. 

If $\sqrt[4]{|A||B||C||D|} > \frac{\sqrt{2}q^{\frac{5}{4}}}{\sqrt{q-1}}$, then $\sqrt{|X| |Y|}=q \sqrt{|A||B||C||D|} > \frac{2q^3 \sqrt{q}}{q-1}$. Hence, it follows from Theorem~\ref{intro thm} that for any $\alpha \in \Bbb F_q^{\ast}$ there exists some $M \in X$ and $N \in Y$ such that $M-N$ has determinant $\alpha$. This implies there exists some $m_{11} \in A$, $n_{11} \in B$, $m_{22} \in C$ and $n_{22} \in D$ such that $(m_{11}-n_{11})(m_{22}-n_{22})=\alpha$.
\end{proof}

\begin{cor}
If $A$ is a subset of $\Bbb F_q$ with cardinality $|A| > \frac{3} {2} q^{\frac{3}{4}}$, then the subset $(A - A) (A - A)$ equals all of $\Bbb F_q$.
\end{cor}
\begin{proof}
Consider the set of matrices $E=\{B \in \operatorname{Mat}_2(\Bbb F_q) \mid b_{11}, b_{22} \in A \text{ and } b_{21}=0\}$. 

First notice that if $q \leqslant 9$ the result is easy to show. If $q>9$, then $|A| > \frac{3}{2} q^{\frac{3}{4}}$ implies $|E|=q |A|^2 > \frac{2q^3 \sqrt{q}}{q-1}$.

If $|E|=q  |A|^2 > \frac{2q^3 \sqrt{q}}{q-1}$, then it follows from Theorem~\ref{intro thm} that for any $\alpha \in \Bbb F_q^{\ast}$ there exists two matrices in $E$ with a difference matrix of determinant $\alpha$. This implies any $\alpha \in \Bbb F_q$ can be written as $(b_{11}-c_{11})(b_{22}-c_{22})=\alpha$ for some matrices $B,C \in E$ and the result follows.
\end{proof}

\appendix
\section{On Spectral Theory Of Cayley Digraphs} 

Let $H$ be a finite abelian group and $S$ be a subset of $H$. The \emph{Cayley digraph} $\operatorname{Cay}(H,S)$ is the digraph whose vertex set is $H$ and there is an edge from $u$ to $v$ (denoted by $u \rightarrow v$) if and only if $v-u \in S$. By definition $\operatorname{Cay}(H,S)$ is a simple digraph with $d^{+}(u)=d^{-}(u)=|S|$. Furthermore, if we have a Cayley digraph, then we can find its spectrum easily using characters in representation theory, see \cite{Serre} for a rigorous treatment on character theory. A function $\chi: H \longrightarrow \mathbb{C}$ is a \emph{character} of $H$ if $\chi$ is a group homomorphism of $H$ into the multiplicative group $\mathbb{C}^{\ast}$. If $\chi(h)=1$ for every $h\in H$, we say $\chi$ is the \emph{trivial character}. The following theorem is a very important well-known fact, see e.g. \cite{Brouwer-Spectra}: 

\begin{thm} \label{character theory}
Let $\mathbb{A}$ be an adjacency matrix of a Cayley digraph $\operatorname{Cay}(H,S)$. Let $\chi$ be a character on $H$. Then the vector $(\chi(h))_{h \in H}$ is an eigenvector of $\mathbb{A}$, with eigenvalue $\sum_{s \in S} \chi(s)$. In particular, the trivial character corresponds to the trivial eigenvector $\boldsymbol{1}$ with eigenvalue $|S|$.
\end{thm}

\begin{proof} 
Let $u_1,u_2, \cdots, u_n$ be an ordering of the vertices of the digraph and let $\mathbb{A}$ correspond to this ordering. Pick any $u_i$. Then we have \[ \sum_{j=1}^{n} \mathbb{A}_{ij} \chi(u_j)= \sum_{u_{i} \rightarrow u_{j}} \chi(u_{j})=\sum_{s \in S} \chi(u_{i}+s)= \sum_{s \in S} \chi(u_{i})\chi(s)=\bigg[\sum_{s \in S} \chi(s)\bigg]\chi(u_{i}).\]
\end{proof}

Notice that we get $|H|$ many eigenvectors as demonstrated in the previous theorem, and they are all distinct since they are orthogonal by \emph{character orthogonality}, see \cite{Serre}. This means we know all of the eigenvectors of a Cayley digraph explicitly assuming we know all of the characters of $H$. 

Theorem~\ref{SGT} (viz. spectral gap theorem) below is a very important and widely used tool in graph theory by itself, see \cite{Demiroglu1} for the proof.
\begin{thm}[Spectral Gap Theorem For Cayley Digraphs] \label{SGT}
Let $Cay(H,S)$ be a Cayley digraph of order $n$. Let $\{ \chi_i \}_{i=1,2, \cdots, n}$ be the set of all distinct characters on $H$ such that $\chi_{1}$ is the trivial one. Define \[ n_{\ast}= \frac{n}{|S|} \left( \max_{2 \leqslant i \leqslant n} \Big\| \sum_{s \in S} \chi_{i}(s) \Big\| \right) \] and let $X,Y$ be subsets of vertices of $Cay(H,S)$. If $\sqrt{|X||Y|} > n_{\ast}$, then there exists a directed edge between a vertex in $X$ and a vertex in $Y$. In particular if $|X|> n_{\ast}$, then there exists at least two distinct vertices of $X$ with a directed edge between them.
\end{thm}


\begin{thebibliography}{10}



\bibitem{Brouwer-Spectra}
Brouwer, Andries E., and Willem H. Haemers.
\newblock {\em Spectra of graphs}.
\newblock Springer Science \& Business Media, 2011.


\bibitem{Demiroglu1}
Dem\.iro\u{g}lu Karabulut, Ye\c{s}\.im. 
"Waring's Problem in Finite Rings." 
\newblock {\em arXiv preprint arXiv:1709.04428} (2017).


\bibitem{Farb-Algebra}
Farb, Benson, and R. Keith Dennis.
\newblock {\em Noncommutative algebra.} Vol. 144.
\newblock Springer Science \& Business Media, 2012.




\bibitem{Godsil}
Godsil, Chris, and Gordon F. Royle.
\newblock {\em Algebraic graph theory.} Vol. 207. 
\newblock Springer Science \& Business Media, 2013.


\bibitem{Henriksen}
Henriksen, Melvin. 
"Two classes of rings generated by their units." 
\newblock {\em Journal of Algebra} 31, no. 1 (1974): 182-193.


\bibitem{Hungerford}
Hungerford, Thomas W. 
\newblock {\em Algebra.} Vol. 73.
\newblock Springer Science \& Business Media, 2003.


\bibitem{Murphy}
Iosevich, A., B. Murphy, and J. Pakianathan. 
"The square root law and structure of finite rings." 
\newblock {\em arXiv preprint arXiv:1405.7657} (2014).


\bibitem{Serre}
Serre, Jean-Pierre.
\newblock {\em Linear representations of finite groups.} Vol. 42.
\newblock Springer Science \& Business Media, 2012.

 
\bibitem{Weil}
Weil, André.
"On some exponential sums."
\newblock {\em Proceedings of the National Academy of Sciences} 34, no. 5 (1948): 204-207.


\bibitem{West}
West, Douglas Brent.
\newblock {\em Introduction to graph theory.} Vol. 2.
\newblock Upper Saddle River: Prentice hall, 2001.


\bibitem{Wolfson}
Wolfson, Kenneth G.
"An ideal-theoretic characterization of the ring of all linear transformations."
\newblock {\em American Journal of Mathematics} 75, no. 2 (1953): 358-386.


\bibitem{Zelinsky}
Zelinsky, Daniel. 
"Every linear transformation is a sum of nonsingular ones." 
\newblock {\em Proceedings of the American Mathematical Society} 5, no. 4 (1954): 627-630.





\end{thebibliography}
\end{document}